\documentclass[12pt]{amsart}
\usepackage{amssymb}
\usepackage[all]{xy}
\usepackage{shadow}
\usepackage{enumerate}

\numberwithin{equation}{section}

\newtheorem{theorem}{Theorem}[section]
\newtheorem{proposition}[theorem]{Proposition}
\newtheorem{lemma}[theorem]{Lemma}
\newtheorem{corollary}[theorem]{Corollary}

\theoremstyle{definition}
\newtheorem{definition}[theorem]{Definition}

\theoremstyle{remark}
\newtheorem{remark}[theorem]{Remark}

\newcommand{\ShortExactSeq}[3]%
{0\to{#1}\to{#2}\to{#3}\to 0}
\newcommand{\DShortExactSeq}[3]%
{\xymatrix@1{0\ar[r]&{{#1}}\ar[r]&{{#2}}\ar[r]&{{#3}}\ar[r]&0}}

\newcommand{\Bounded}{{\frak B}}

\newcommand{\Compacts}{{\frak K}}

\newcommand{\C}{\mbox{$\Bbb C$}}

\newcommand{\prop}{\operatorname{\rm Prop}}

\newcommand{\N}{\mbox{$\Bbb N$}}

\newcommand{\R}{\mbox{$\Bbb R$}}

\newcommand{\Supp}{\operatorname{\rm Supp}}

\newcommand{\diam}{\operatorname{\rm diam}}

\renewcommand{\ge}{\geqslant}
\renewcommand{\le}{\leqslant}
 \renewcommand{\phi}{\varphi}
\renewcommand{\epsilon}{\varepsilon}
\renewcommand{\leq}{\leqslant}
\renewcommand{\geq}{\geqslant}

 \newcommand{\hlf}{{\textstyle\frac12}}
\newcounter{ritmctr}
{\end{itemize}}
\newcounter{aitmctr}
{\end{itemize}}

\begin{document}

\title{Ghostbusting and property A}

\author{John Roe}
\address{Department of Mathematics, Penn State University, University
Park PA 16802}
\email{john.roe@psu.edu}
\author{Rufus Willett}
\address{Department of Mathematics, University of Hawai'i at M\={a}noa, 2565 McCarthy Mall, Honolulu, Hawai'i 96822}
\email{rufus@math.hawaii.edu}
\date{\today}
\maketitle

\section{Introduction}
Let $X$ be a metric space (we may allow $+\infty$ as a value for some distances in $X$).  We say that $X$ has \emph{bounded geometry} if, for each $R>0$, there is a natural number $N$ such that every ball of radius $R$ in $X$ contains at most $N$ points.  (In particular, $X$ is discrete.)  In this paper, we will consider bounded geometry metric spaces in this sense.

Let  $X$ be such a space, and let $\ell^2(X)$ denote the usual Hilbert space of square summable functions on $X$ with fixed orthonormal basis $\{\delta_x~|~x\in X\}$ of Dirac masses.  Let $\Bounded(\ell^2(X))$ denote the $C^*$-algebra of bounded operators on $\ell^2(X)$.  If $T$ is an element of $\Bounded(\ell^2(X))$, then $T$ can be uniquely represented as an $X$-by-$X$ matrix $(T_{xy})_{x,y\in X}$, where
$$
T_{xy}=\langle \delta_x,T\delta_y\rangle.
$$

The following definitions are standard~\cite{roe_lectures_2003}.
\begin{definition}\label{uniroe}
If $T$ is an element of $\Bounded(\ell^2(X))$, then the \emph{propagation} of $T$ is defined to be
$$
\prop(T)=\sup\{d(x,y)~|~T_{xy}\neq 0\}.
$$
For each $R\geq 0$ let $\C_R[X]$ denote the collection of all operators of propagation at most $R$, and define
$$
\C_u[X]:=\cup_{R\in [0,\infty)}\C_R[X];
$$
it is not difficult to see that this is a $*$-subalgebra of $\Bounded(\ell^2(X))$.
We let $C^*_u(X)$ denote its norm closure, a $C^*$-algebra called the \emph{translation $C^*$-algebra} or \emph{uniform Roe algebra} of $X$.
\end{definition}

\begin{definition}\label{ghosts} (Guoliang Yu)
An operator $T$ in $C^*_u(X)$ is called a \emph{ghost} if $T_{xy}\to 0$ as $x,y\to\infty$ in $X$.  We denote by $G^*(X)$ the collection of all ghost operators, which is an ideal in $C^*_u(X)$ containing the compact operators $\Compacts$.
\end{definition}

In~\cite{yu_coarse_2000}, Yu introduced \emph{property A}, an amenability property of metric spaces which has since been intensively studied (see \cite{willett_notes_2009} for a survey).  It is easy to prove that for a property A space of bounded geometry, all ghost operators are compact.  Our main objective in this paper is to show

\begin{theorem}\label{ghost-theorem} A bounded geometry metric space without property A always admits non-compact ghosts.  That is, property A is equivalent to the property ``all ghosts are compact''. \end{theorem}

The first examples of non-compact ghosts were \emph{projection} operators  on box spaces arising from residually finite property T groups.  As a corollary of our work, we see

\begin{theorem}\label{noghost-proj} For the bounded geometry metric space constructed by Arzhantseva, Guentner and \v{S}pakula \cite{arzhantseva_coarse_2012}, there exist non-compact ghosts, but all ghost \emph{projections} are compact.  \end{theorem}

This example embeds coarsely in Hilbert space and therefore satisfies the coarse Baum-Connes conjecture, by the work of Yu.  The compactness of ghost projections is a consequence of this.

An outline of the paper is as follows.  In Section 2 we review the definition of property A and of two related properties, the \emph{operator norm localization} (ONL) property of \cite{chen_metric_2008} and the \emph{uniform local amenability} (ULA) property of \cite{brodzki_uniform_2012}. It is known that ONL is equivalent to property A, and ULA is implied by property A.  In Section 3 we prove that the failure of ULA implies the existence of non-compact ghosts, and in Section 4 we follow a similar argument to show that the failure of ONL implies the same result.  As the reader will perceive, Section 3 is therefore logically redundant, but it enables us to introduce the main idea of the proof in a more geometrically natural context.  Finally, in Section 5 we investigate the existence of non-compact ghost projections.

The first author is grateful for the hospitality of the University of Hawai'i during February, 2013, which made this work possible.

\section{Metric amenability properties}
Like other versions of amenability, property A has numerous equivalent formulations.  Here is a convenient one in terms of positive definite kernels.

\begin{definition} A bounded geometric metric space $X$ has \emph{property A} if there exists an sequence $k_n$ of positive definite kernels on $X\times X$ such that
\begin{enumerate}[(a)]
\item Each kernel has controlled support: that is, of each $n$ there is an $r$ such that $k_n(x,y)=0$ whenever $d(x,y)>r$.
\item The sequence $\{k_n\}$ tends to the constant 1 uniformly on each controlled set: that is, for each $s>0$ and $\epsilon>0$ there is $N$ such that $1-\epsilon \le k_n(x,y)\le 1$ for all $n>N$ and $x,y$ such that $d(x,y)<s$.
\end{enumerate}
 \end{definition}

It is this definition of property A that is used in the proof of the following well-known result.

\begin{proposition} \cite[Proposition 11.43]{roe_lectures_2003} On a space with property A, all ghosts are compact.\quad\qedsymbol \end{proposition}

We will not make further direct use of the definition of property A; instead, for the remainder of the paper  we will proceed via two related properties.  The first of these, the \emph{operator norm localization property}, was introduced by Chen, Tessera, Wang, and Yu \cite[Section 2]{chen_metric_2008}.  Later, Sako~\cite{sako_property_2012} showed that this property is equivalent to property A.

\begin{definition}\label{onl}
The space $X$ has the \emph{operator norm localization property} (ONL) if for all $R\geq 0$ and $c\in (0,1)$ there exists $S>0$ such that for all $T\in \C_R[X]$ of norm one there exists a norm one element $\xi\in \ell^2(X)$ such that
$$
\diam(\Supp(\xi))\leq S ~~\text{ and }~~\|T\xi\|\geq c.
$$
\end{definition}

A priori, this definition of ONL  is more restrictive than the original one \cite[Definition 2.3]{chen_metric_2008}, but they are equivalent by~\cite[Proposition 3.1]{sako_property_2012}.

The second property, \emph{uniform local amenability}, was introduced by Brodzki et~al.~\cite{brodzki_uniform_2012}.  Since (as explained above) a full discussion of this property is not logically necessary to our argument, we will reformulate it in a way that is more suitable to our purposes.

\begin{definition}\label{weak-expander}
Let $(X_n)$ be a sequence of non-empty finite metric spaces, and let $X=\sqcup X_n$ be the disjoint union equipped   with a metric that restricts to the given metric on each $X_n$, and is such that $d(X_n,X_m)> \diam(X_n)+\diam(X_m)$ when $n\neq m$.  We assume that $X$ has bounded geometry.  For each $R>0$, let
$$
E_n^R=\{(x,y)\in X_n\times X_n~|~d(x,y)\leq R\}.
$$
We say that $X$ is a \emph{weak expander} if there exist $c,R>0$ such that for all $S>0$ there exists $N$ such that for all $n\geq N$ and all $\phi\colon X_n\to \R$ supported in a ball of radius $S$ we have that
$$
\sum_{(x,y)\in E_n^R}|\phi(x)-\phi(y)|\geq c \sum_{x\in X_n}|\phi(x)|.
$$
\end{definition}

Examples include box spaces of non-amenable groups, sequences of graphs with all vertices of degree at least $3$ and girth tending to infinity \cite{willett_higher_2012}, and expanders.   In particular, the coarsely embeddable, but not property A, box space of Arzhantseva, Guentner and \v{S}pakula \cite{arzhantseva_coarse_2012} is an example.

\begin{remark} An essentially equivalent\footnote{The only difference between the two definitions is very minor: Sako's `boxes' $X_n$ are at infinite distance from each other, and ours are at finite-but-increasing distance.}  definition of ``weak expander'' is given by Sako~\cite[Definition 2.4]{sako_generalization_2012}. \end{remark}

\begin{definition} A bounded geometry metric space $X$ is \emph{uniformly locally amenable} (ULA) if it has no subspace which is a weak expander. \end{definition}

This is a reformulation of the definition of~\cite{brodzki_uniform_2012} (the equivalence of the two definitions can be proved by methods similar to those of Section~4, but we will simply use the definition above.)  It is proved in~\cite{brodzki_uniform_2012} that property A implies uniform local amenability; in particular, no weak expander can have property A.

\section{Ghosts from weak expanders}
In this section we will prove that a bounded geometry space that is not uniformly locally amenable has non-compact ghosts.  Evidently, a ghost on a subspace of a metric space $X$ gives rise (by ``extension by zero'') to a ghost on the whole space.  It therefore suffices to prove

\begin{proposition}\label{ghostthe}
If $X$ is a weak expander, then $C^*_u(X)$ contains non-compact ghost operators.
\end{proposition}

To motivate the argument below, consider the standard example of a space with non-compact ghosts, namely a box space $\Box G$ associated to a residually finite, property T group $G$.  The image of the Kazhdan projection under the natural homomorphism $C^*(G) \to C^*_u(\Box G)$ is a non-compact ghost.  This ghost can also be regarded as the orthogonal projection on the kernel of the natural graph Laplacian $\Delta$ on $\Box G$, and since property $T$ implies that the Laplacian has a spectral gap, we can also write this projection as $f(\Delta)$ for a suitable function $f$ supported near zero.   This motivates the search for ghosts on weak expanders which also have the form $f(\Delta)$ for $f$ supported near zero.

Let $X$ be a bounded geometry metric space and $R>0$.  Recall that the \emph{Laplacian at scale $R$} is the operator $\Delta=\Delta_R$ on $\ell^2(X)$ defined by
$$
\Delta_R\colon \delta_x\mapsto \sum_{y:(x,y)\in E_n^R}(\delta_x-\delta_y),
$$
where the notation $E_n^R$ has the same significance as in Definition~\ref{weak-expander}.
The operator $\Delta_R$ has propagation $R$, and is bounded (since $X$ has bounded geometry); in particular $\Delta_R$ is an element of $C^*_u(X)$.  A straightforward computation shows that for any $\xi\in \ell^2(X)$ we have
$$
\langle \Delta_R\xi,\xi\rangle=\hlf\sum_{(x,y)\in E_n^R}|\xi(x)-\xi(y)|^2,
$$
and thus in particular that $\Delta_R$ is a positive operator.
We have the following lemma.

\begin{lemma}\label{laplem}  Suppose that $X$ is a weak expander.
Then there exist $R>0$ and $\kappa>0$ such that for any $S>0$ there exists $N$ such that for any $n\geq N$, and any norm one $\xi\in \ell^2(X_n)$ with support in a ball of radius $S$ we have
$$
\langle \Delta_R \xi,\xi\rangle \geq \kappa.
$$
\end{lemma}

\begin{proof}
Let $N$ be as in~Definition~\ref{weak-expander} for the parameter $S$.  Let $\xi$ be as in the statement and define $\phi\colon X_n\to \R$ by
$$
\phi(x)=|\xi(x)|^2.
$$
The definition of weak expander implies that
$$
\sum_{(x,y)\in E_n^R}|\phi(x)-\phi(y)|\geq c \sum_{x\in X_n}|\phi(x)|,
$$
i.e.\ that
\begin{equation}\label{firstinq}
\sum_{(x,y)\in E_n^R}||\xi(x)|^2-|\xi(y)|^2|\geq c \sum_{x\in X_n}|\xi(x)|^2=c.
\end{equation}
Looking at the left hand side above, we have
\begin{align*}
\sum_{(x,y)\in E_n^R}& ||\xi(x)|^2-|\xi(y)|^2| =\sum_{(x,y)\in E_n^R}|(|\xi(x)|-|\xi(y)|)(|\xi(x)|+|\xi(y)|)| \\
&\leq\sqrt{ \sum_{(x,y)\in E_n^R}(|\xi(x)|-|\xi(y)|)^2}\sqrt{\sum_{(x,y)\in E_n^R}(|\xi(x)|+|\xi(y)|)^2} \\
& \leq \sqrt{\sum_{(x,y)\in E_n^R}|\xi(x)-\xi(y)|^2}\sqrt{\sum_{(x,y)\in E_n^R}2|\xi(x)|^2+2|\xi(y)|^2} \\
&\leq\sqrt{2\langle \Delta\xi,\xi\rangle} \sqrt{4M},
\end{align*}
where $M$ is a bound on the size of balls in $X$ of radius $R$.  Comparing this to line \eqref{firstinq} gives the desired statement with $\kappa=c^2/8M$.
\end{proof}

\begin{proof}[Proof of Proposition \ref{ghostthe}]   Let $X$ be a weak expander, and let $R$ and $\kappa$ be the quantities provided by Lemma~\ref{laplem}.  We abbreviate $\Delta_R$ as $\Delta$, and put $m=\|\Delta\|$.
Let $f\colon\R^+\to[0,1]$ be any continuous function with support in $[0,\kappa/2]$, and such that $f(0)=1$.    Since $f(\Delta)$ majorizes $\chi_{0}(\Delta)$, which is the orthogonal    projection onto the infinite-dimensional space of ``$R$-locally constant'' functions,  it is clear that $f(\Delta)$ is not compact.  It suffices therefore to show that $f(\Delta)$ is a ghost.

Let $\epsilon>0$, and let $p$ be a polynomial such that
$$
\sup_{x\in [0,m]}|f(x)-p(x)|<\epsilon.
$$
Note that the propagation of $p(\Delta)$ is at most $\deg (p) \cdot R$, and that $\|f(\Delta)-p(\Delta)\|<\epsilon$ by the spectral theorem.  Let $S>\deg(p) \cdot R$, and let $N$ be as Lemma \ref{laplem} with respect to this $S$.  (We may assume that $N$ is large enough that any ball of radius $S$ whose center lies in some $X_n$, $n\ge N$, is itself a subset of that $X_n$.)  Let $x$ be a point in $X_n$ for some $n\geq N$.  We will show that
$$
\|p(\Delta)\delta_x\|<\epsilon,
$$
whence $\|f(\Delta)\delta_x\|<2\epsilon$.  It follows that all the matrix entries $\langle f(\Delta)\delta_x,\delta_y\rangle$ are bounded by $2\epsilon$ whenever $x$ (or $y$) lies in $B_n$ for $n\geq N$.  Since  $\epsilon$ is arbitrary, this will prove that $f(\Delta)$ is a ghost.

Consider then $B=B(x;S)\subseteq X_n$.  Let $P\colon \ell^2(X)\to \ell^2(B)$ be the orthogonal projection onto $\ell^2(B)$, and consider the operator $T=P\Delta P$ as an operator on $\ell^2(B)$.  Lemma~\ref{laplem} implies that the operator $T$ is strictly positive, with spectrum contained in $[\kappa,m]$.  It follows that $f(T)=0$, and thus that $\|p(T)\|<\epsilon$.  On the other hand, for any $k\leq \deg(p)$ the vector $\Delta^k\delta_x
$
is supported in $B$, whence $\Delta^k\delta_x=T^k\delta_x$ for all such $k$, and thus also
$$
p(\Delta)\delta_x=p(T)\delta_x.
$$
Hence finally
$$
\|p(\Delta)\delta_x\|=\|p(T)\delta_x\|\leq \|p(T)\|<\epsilon,
$$
completing the proof.
\end{proof}

\begin{remark} This argument is inspired by the proof of the ``partial vanishing theorem'' of~\cite{roe_positive_2012}. \end{remark}

\section{Ghosts if ONL fails}

In this section we will adapt the argument of Section~3 to construct non-compact ghosts for any bounded geometry metric space $X$ that does not have the operator norm localization property.  Since Sako has proved the equivalence of ONL and property~A, this will complete the proof of our main result, Theorem~\ref{ghost-theorem}.

We note for future reference

\begin{lemma}\label{union-lemma} The operator norm localization property passes to finite unions: if $X=Y\cup Z$, and both $Y$ and $Z$ have ONL, then  so does $X$. \end{lemma}

\begin{proof} This is a special case of Lemma~3.3 in~\cite{chen_operator_2009}.  (Of course, granted that ONL is equivalent to property A, it also follows from the corresponding observation for property A \cite{dadarlat_uniform_2007}.)  \end{proof}

The following technical lemma builds a useful sequence of operators from the failure of the  operator norm localization property.

\begin{lemma}\label{onlprop}
Let $X$ be a bounded geometry metric space that does not have ONL.  Then there exist $R>0$, $\kappa<1$, a sequence  $(T_n)$ of operators in $\C_R[X]$, a sequence $(B_n)$ of finite subsets of $X$, and a sequence $(S_n)$ of positive real numbers such that:
\begin{enumerate}[(a)]
\item $(S_n)$ is an increasing sequence tending to $\infty$ as $n$ tends to $\infty$;
\item each $T_n$ is positive and of norm one;
\item for $n\neq m$, $B_n\cap B_m=\emptyset$;
\item if $P_n\colon \ell^2(X)\to \ell^2(B_n)$ denotes the orthogonal projection, then $P_nT_nP_n=T_n$;
\item for each $n$, for all $\xi\in \ell^2(X)$ satisfying
$$ \|\xi\|=1,\quad
\diam(\Supp(\xi))\leq S_n,
$$
we have
$$
\|T_n\xi_n\|\leq \kappa.
$$
\end{enumerate}
\end{lemma}

For simplicity, we assume in the proof that the metric on $X$ only takes finite values: the general case can be treated similarly.

\begin{proof}
Fix a basepoint $x_0$ in $X$, and let
$$
Y=\bigsqcup _{ m\text{ even} } \{x\in X~|~m^2\leq d(x_0,x)\leq (m+1)^2\}
$$
and
$$
Z=\bigsqcup _{ ~m\text{ odd} } \{x\in X~|~m^2\leq d(x_0,x)\leq (m+1)^2\}.
$$
We have then that $X=Y\cup Z$.  By Lemma~\ref{union-lemma},  either $Y$ or $Z$ does not have ONL; say without loss of generality $Y$ does not have ONL.

Now, the   negation of ONL for $Y$ implies that there exist $R>0$ and $c<1$ such that for any $S>0$ there exists a norm one operator $T\in \C_R[X]$ such that
\begin{equation}\label{smallnorm}
\diam(\Supp(\xi))\leq S \text{ and } \|\xi\|=1 ~~\text{ implies }~~\|T\xi\|< c.
\end{equation}
We will call such an operator \emph{$(R,c,S)$-localized}.  In fact, on replacing $T$ by $T^*T$ (and $R$ by $2R$ and $c$ by $\sqrt{c}$), we see that
there exist $R>0$ and $c>1$ such that for every $S>0$ there exists a \emph{positive} norm one operator which is $(R,c,S)$-localized.  For the remainder of the proof, let $R$ and $c$ denote fixed quantities with this property, and let $\kappa=2c/(1+c)<1$.

Note that $Y$ is a generalized box space --- that is, a disjoint union of finite components, with the distance between components tending to infinity. It follows that  there exists an \emph{$R$-separated decomposition}
$$
Y=\sqcup_{m=1}^\infty Y_m
$$
where each $Y_m$ is a non-empty finite subset of $Y$ such that for $n\neq m$, $d(Y_n,Y_m)>R$. In particular, any $T\in \C_R[X]$ splits as a block diagonal sum of finite rank operators $T=\oplus_m T^{(m)}$, $T_m\in \Bounded(\ell^2(Y_m))$, with respect to this decomposition.

We now define $(T_n)$, $(S_n)$ and $(B_n)$ inductively as follows.  Suppose that these sequences have already been defined for $n<N$.  Choose   $M$   so large that
$$
\bigcup_{n=1}^{N-1}B_n\subseteq \bigsqcup_{m\leq M} Y_m
$$
and choose $S_N$   so large that $S_N\geq n$, $S_N>S_{N-1}$ and
$$
S_N>\diam\Big(\bigsqcup_{m\leq M} Y_m\Big).
$$
(In the base case $N=1$, we simply set $S_1=1$.)  Choose a positive norm one operator
  $T\in\C_R[X]$ which is $(R,c,S_N)$-localized~(\eqref{smallnorm}).   As
$$
\|T\|=\sup_{m\in\N}\|T^{(m)}\|_{\Bounded(\ell^2(Y_m))}
$$
there exists $m\in\N$ such that $\|T^{(m)}\|>\hlf (1+c)$.  (In particular this forces $m>M$.) Set
$$
T_N=\frac{T^{(m)}}{\|T^{(m)}\|},
$$
and note that for any $\xi\in \ell^2(Y)$ of norm one and with $\diam(\Supp(\xi))\leq S_N$ we have that
$$
\|T_N\xi\|\leq \frac{2c}{1+c}=\kappa<1.
$$
Set $B_N=Y_m$.

Assume then that $(T_n)_{n\leq N-1}$, $(S_n)_{n\leq N-1}$ and $(B_n)_{n\leq N-1}$ have been defined.
Let $T$ be a positive norm one operator in $\C_R[X]$ with the property in line \eqref{smallnorm} for $S=S_N$.  Let $m$ be such that $\|T^{(m)}\|>\frac{1}{2}(1+c)>c$, and note that by choice of $S_N$, this forces $m>M$.  Set
$$
T_N=\frac{T^{(m)}}{\|T^{(m)}\|},
$$
and let $B_N=Y_m$.
This completes the inductive construction.
\end{proof}

\begin{proof} (Proof of Theorem~\ref{ghost-theorem}.)
Let $X$ be a bounded geometry space without property A (or, equivalently, without ONL).  Let $R>0$, $\kappa<1$, and sequences  $(T_n)$,   $(B_n)$ and $(S_n)$ be constructed as in Lemma~\ref{onlprop} above.
It follows from the   construction that
$$
T=\oplus T_n
$$
is a positive norm one operator in $\C_R[X]$.  Let now $f\colon [0,1]\to[0,1]$ be any continuous function supported in $[(1+\kappa)/2,1]$ such that $f(1)=1$, and let $f(T)\in C^*_u(X)$ denote the element given by the functional calculus.  The operator $f(T)$ is positive, norm one, and decomposes as a block diagonal sum
$$
f(T)=\oplus f(T_n)
$$
where each $f(T_n)$ comes from an operator on $\ell^2(B_n)$.

We claim that operator $f(T)$ so constructed  is a non-compact ghost operator. To see this, note first that as each $T_n$ is a positive, norm one, finite rank operator, 1 is an eigenvalue of $T_n$.  It follows that $\chi_{\{1\}}(T)$ (defined using the Borel functional calculus) is an infinite rank projection; as $f(T)\geq \chi_{\{1\}}(T)$, this implies that $f(T)$ is non-compact.

It thus remains to show that $f(T)$ is a ghost.  We argue as in the proof of Proposition~\ref{ghostthe}.  It will suffice to show that for any $\epsilon>0$ there exists $N$ such that if $n\geq N$ and $x\in B_n$, then $\|f(T_n)\delta_x\|\leq 2\epsilon$.

Let  $p$ be a polynomial such that
$$
\sup_{x\in [0,1]}|f(x)-p(x)|<\epsilon.
$$
Note that the propagation of $p(T)$ is at most $\deg(p) \cdot R$, and that
$$
\|f(T)-p(T)\|=\sup_n\|f(T_n)-p(T_n)\|<\epsilon
$$
by the spectral theorem.  Let $N$ be so large that $S_n>2\deg(p) \cdot R$ for all $n\geq N$.  Let $x$ be a point in $B_n$ for some $n\geq N$.  We will show that
$$
\|p(T_n)\delta_x\|<\epsilon,
$$
whence $\|f(T_n)\delta_x\|<2\epsilon$ as required.

Consider then $B=B(x;\hlf S_N )$.  Let $P\colon \ell^2(X)\to \ell^2(B)$ be the orthogonal projection onto $\ell^2(B)$, and consider the (positive) operator $T_n'=PT_n P$.  The fact that $\diam(B)\leq S_N\leq S_n$ implies that $\|T_n'\|\leq c$ whence the spectrum of $T_n'$ is contained in
 $[0,c]$.  It follows that $f(T_n')=0$, and thus that $\|p(T_n')\|<\epsilon$.  On the other hand, for any $k\leq \deg(p)$ we have that $T_n^k\delta_x$
is supported in $B$, whence $T_n^k\delta_x=(T_n')^k\delta_x$ for all such $k$, and thus also
$$
p(T_n)\delta_x=p(T_n')\delta_x.
$$
Hence finally
$$
\|p(T_n)\delta_x\|=\|p(T_n')\delta_x\|\leq \|p(T_n')\|<\epsilon,
$$
completing the proof of Theorem~\ref{ghost-theorem}.
\end{proof}

\begin{remark}
Let $X$ be a bounded geometry metric space without property A.  Then the above construction gives rise to a ghost operator $W=f(\Delta)\in C^*_u(X)$ that splits as a block diagonal sum
$\oplus_{n}W_n$
of norm one operators.  Thus, for any two distinct  subsets $E,F$ of $\N$ the operators
$$
\bigoplus_{n\in E}W_n,\quad\bigoplus_{n\in F}W_n
$$
are also ghosts at distance one from each other.  It follows  that as soon as the ghost ideal $G^*(X)$ is not equal to the compact operators, it is not separable.
\end{remark}

\section{Additional remarks}

\subsection{Exact groups}  Let $G$ be a discrete group (which we assume to be finitely generated in order to make contact with the metric space language of this paper).  Then Ozawa~\cite{ozawa_amenable_2000} and Guenter-Kaminker~\cite{guentner_exactness_2002} showed that the underlying coarse space of $G$ has property $A$ if and only if $G$ is \emph{exact}: that is, for any short exact sequence of $G$-$C^*$-algebras
\[ 0 \to I \to A \to B \to 0, \]
the corresponding sequence of (reduced) cross product algebras
\[ 0 \to I\rtimes_r G \to A\rtimes_r G \to B\rtimes_r G \to 0 \]
is exact also.

Consider in particular the sequence
\begin{equation}\label{spec-seq}
 0 \to c_0(G) \to \ell^\infty(G) \to (\ell^\infty(G)/c_0(G)) \to 0
 \end{equation}
of \emph{commutative} $G$-$C^*$-algebras.  Taking the cross product with $G$ we obtain the sequence
\[ 0 \to \Compacts \to C^*_u(|G|) \to Q \to 0 , \]
where $Q = (\ell^\infty(G)/c_0(G))\rtimes_r G$.  Moreover, the first author observed in \cite{roe_band-dominated_2005} that the kernel of the surjection $C^*_u(|G|)\to Q$ is precisely the ghost ideal.  Thus property A is equivalent to the statement that the cross product with $G$ preserves exactness for \emph{all} exact sequences of $G$-$C^*$-algebras, and ``all ghosts are compact'' is equivalent to the statement that cross product with $G$ preserves exactness for the single example of Equation~\ref{spec-seq}.  Our main result therefore implies

\begin{corollary} If crossed product with $G$ preserves the exactness of the sequence given by Equation~\ref{spec-seq}, then $G$ is an exact group. \quad\qedsymbol \end{corollary}

\subsection{Ghost projections}

 Non-compact ghost projections are the only known source of counterexamples to the coarse Baum-Connes conjecture (in the bounded geometry setting).  It is possible for $G^*(X)$ to contain non-compact operators but not to contain any non-compact projections, however.  Indeed, let $X$ denote the box space constructed by Arzhantseva, Guentner and \v{S}pakula~\cite{arzhantseva_coarse_2012}.  This space has bounded geometry and   coarsely embeds into Hilbert space, but does not have property A.  It follows from known results on $K$-theory in \cite[Theorem 6.1]{willett_higher_2012} and \cite[Theorem 1.1]{yu_coarse_2000} that for this $X$, the algebra $C^*_u(X)$ contains no non-compact ghost projections\footnote{This also follows from recent results of Finn-Sell \cite[Corollary 35]{finn-sell_martin_fibered_2013}.}, despite containing non-compact ghosts by the results of this paper.  This is Theorem~\ref{noghost-proj}.


\providecommand{\bysame}{\leavevmode\hbox to3em{\hrulefill}\thinspace}
\providecommand{\MR}{\relax\ifhmode\unskip\space\fi MR }
\providecommand{\MRhref}[2]{%
  \href{http://www.ams.org/mathscinet-getitem?mr=#1}{#2}
}
\providecommand{\href}[2]{#2}

\end{document}